\documentclass[11pt,leqno]{article}

\usepackage[T1]{fontenc}
\usepackage[utf8]{inputenc}
\usepackage[english]{babel}
\usepackage{lmodern}
\usepackage{microtype}
\usepackage[margin=1.05in]{geometry}
\usepackage{mathtools}
\usepackage{amssymb,amsmath,amsthm}
\usepackage{hyperref}
\usepackage{aliascnt}
\usepackage[nameinlink,noabbrev]{cleveref}

\newcommand{\abs}[1]{\left\lvert #1\right\rvert}

\newcommand{\logtwo}{\log_2}
\newcommand{\PP}{\mathcal{P}}

\theoremstyle{plain}
\newtheorem{theorem}{Theorem}[section]

\newaliascnt{lemma}{theorem}

\aliascntresetthe{lemma}

\newaliascnt{corollary}{theorem}
\newtheorem{corollary}[corollary]{Corollary}
\aliascntresetthe{corollary}

\newaliascnt{proposition}{theorem}
\newtheorem{proposition}[proposition]{Proposition}
\aliascntresetthe{proposition}

\theoremstyle{definition}
\newaliascnt{hypothesis}{theorem}
\newtheorem{hypothesis}[hypothesis]{Hypothesis}
\aliascntresetthe{hypothesis}

\newaliascnt{definition}{theorem}
\newtheorem{definition}[definition]{Definition}
\aliascntresetthe{definition}

\theoremstyle{remark}
\newaliascnt{remark}{theorem}
\newtheorem{remark}[remark]{Remark}
\aliascntresetthe{remark}

\theoremstyle{definition}
\newaliascnt{example}{theorem}
\newtheorem{example}[example]{Example}
\aliascntresetthe{example}

\crefname{theorem}{Theorem}{Theorems}
\Crefname{theorem}{Theorem}{Theorems}
\crefname{lemma}{Lemma}{Lemmas}
\Crefname{lemma}{Lemma}{Lemmas}
\crefname{corollary}{Corollary}{Corollaries}
\Crefname{corollary}{Corollary}{Corollaries}
\crefname{proposition}{Proposition}{Propositions}
\Crefname{proposition}{Proposition}{Propositions}
\crefname{hypothesis}{Hypothesis}{Hypotheses}
\Crefname{hypothesis}{Hypothesis}{Hypotheses}
\crefname{definition}{Definition}{Definitions}
\Crefname{definition}{Definition}{Definitions}
\crefname{remark}{Remark}{Remarks}
\Crefname{remark}{Remark}{Remarks}
\crefname{example}{Example}{Examples}
\Crefname{example}{Example}{Examples}

\title{Linear truncation for conditioned prime-factor fibres}
\author{J.~Verwee}
\date{}

\begin{document}
\maketitle

\begin{abstract}
In previous joint work with Tenenbaum, the truncation step $f\mapsto f_R$ in the conditional effective Erd\H{o}s--Wintner theorem on the fibre $\omega(n)=k$ yields, in the continuous case for real strongly additive $f$, a remainder of size $\eta_f(R)^{r/(r+1)}$, where $R$ is the truncation level and $r=k/\logtwo x$. We prove an effective linear truncation lemma showing that, in the central window $\kappa\leqslant r\leqslant 1/\kappa$, this bound improves to the natural linear scale $r\,\eta_f(R)$ under an effective Sathe--Selberg-type ratio estimate for the fibre. This yields a direct effective sharpening of the truncation step in the previous joint work. The same truncation upgrade also applies to prime-set restrictions, $\Omega$-fibres, and weighted fibres whenever the corresponding ratio estimate is available.
\end{abstract}

\section{A linear truncation lemma}\label{sec:abstract}

This section isolates the truncation step from the previous joint work with Tenenbaum \cite{TV}, namely the step that yields the term $\eta_f(R)^{r/(r+1)}$ in the continuous case. The key point is that, on the fibre $\omega(n)=k$ and on related fibres, the integers for which truncation alters the value of $f$ can be counted more sharply than by a coarse H\"older-type argument. To make this mechanism transparent, we formulate it abstractly in terms of a generic prime-factor fibre and a single ratio estimate. Throughout, for integers $n\geqslant 1$, we write $\log_n x$ for the $n$-fold iterated logarithm, defined recursively by
\[
\log_1 x:=\log x,
\qquad
\log_{n+1}x:=\log\left(\log_n x\right).
\]
In particular, $\log_2 x=\log\log x$ and $\log_3 x=\log\log\log x$. We also write $\PP$ for the set of primes.

\subsection{Truncation tails}\label{subsec:tails}

Let $f$ be a real \emph{strongly additive} function, meaning $f(p^\nu)=f(p)$ for all primes $p$ and
$\nu\geqslant 1$, and $f(mn)=f(m)+f(n)$ whenever $(m,n)=1$. Assume that $f$ satisfies the hypotheses of the Erd\H{o}s--Wintner theorem, that are
\begin{equation}\label{hypEW}
    \sum_{p\in\PP} \frac{\min(1,f(p)^2)}{p}<\infty, \quad \text{and} \quad \sum_{\substack{p\in\PP\\|f(p)|\leqslant 1}} \frac{f(p)}{p} \text{ converges.}
\end{equation}
We also assume the following continuity condition:
\begin{equation}\label{continu}
    \sum_{f(p)\neq 0} \frac{1}{p} = \infty.
\end{equation}
As in \cite{TV}, we fix a \emph{tail function} $\eta_f$ with the following property: $\eta_f(y)$ is
continuous, non-increasing, tends to $0$ as $y\to\infty$, and for every $y>1$,
\begin{equation}\label{eq:eta}
\abs{\sum_{\substack{p>y\\ \abs{f(p)}\leqslant 1}}\frac{f(p)}{p}}\ \leqslant\ \eta_f(y),
\qquad
\sum_{p^\nu>y}\frac{\min\{1,f(p)^2\}}{p^\nu}\ \leqslant\ \eta_f(y).
\end{equation}
Such a choice exists thanks to \eqref{hypEW}.

Given $R\geqslant 3$, define the truncation $f_R$ on prime powers by
\[
f_R(p^\nu):=
\begin{cases}
f(p),& p\leqslant R\ \text{or}\ \abs{f(p)}\leqslant 1,\\
0,& p>R\ \text{and}\ \abs{f(p)}>1,
\end{cases}
\qquad(\nu\geqslant 1),
\]
and extend additively.
The only primes that can change the value of $f$ under truncation are the \emph{large primes on which $f$ is large}:
\[
P_R:=\{p>R:\ \abs{f(p)}>1\}.
\]
Since $\min\{1,f(p)^2\}=1$ for $p\in P_R$, the second inequality in \eqref{eq:eta} implies
\begin{equation}\label{eq:PReta}
\sum_{p\in P_R}\frac{1}{p}\ \leqslant\ \eta_f(R).
\end{equation}

\subsection{A ratio hypothesis on a fibre}\label{subsec:ratio}

Let $h:\mathbb N\to\mathbb N$ be a function. For $x\geqslant 3$ and $k\geqslant 1$, define the fibre and its cardinality
\[
\mathcal F(x;k):=\{n\in\mathbb N:\ n\leqslant x,\ h(n)=k\},
\qquad
N_k(x):=\#\mathcal F(x;k).
\]

\begin{definition}\label{def:admissible}
We call the family $\mathcal F(x;k)$ an \emph{admissible fibre} if it satisfies the following two
structural properties.
First, it is stable under removing one prime factor:
\begin{equation}\label{eq:stab}
n\in\mathcal F(x;k),\ p\mid n,\ n=p^\nu m,\ (m,p)=1
\quad\Longrightarrow\quad
m\in \mathcal F(x/p^\nu;\,k-1).
\end{equation}
Second, it has a simple lower support bound:
\begin{equation}\label{eq:support}
n\in\mathcal F(x;k)\quad\Longrightarrow\quad n\geqslant 2^k.
\end{equation}
\end{definition}

These properties hold for the classical fibre $\mathcal F(x;k)=\{n\leqslant x:\omega(n)=k\}$, and for several natural variants (e.g.\ restriction to a prime set). We work in the central window
\begin{equation}\label{eq:window}
\kappa\leqslant r:=\frac{k}{\logtwo x}\leqslant \frac{1}{\kappa},
\qquad \kappa\in(0,1)\ \text{fixed}.
\end{equation}
The only analytic input needed for the truncation bound is the following Sathe--Selberg-type ratio estimate on an admissible fibre.

\begin{hypothesis}\label{hyp:ratio}
There exist constants $x_0(\kappa)\geqslant 3$ and $C(\kappa)>0$ such that, uniformly for $x\geqslant x_0(\kappa)$,
$k\geqslant 1$ with \eqref{eq:window}, and every integer $m$ with $2\leqslant m\leqslant x/2^{k-1}$ (the range where $N_{k-1}(x/m)$ may be nonzero by \eqref{eq:support}), one has
\begin{equation}\label{eq:ratio}
N_{k-1}\left(\frac{x}{m}\right)\ \leqslant\ C(\kappa)\,\frac{r}{m}\,N_k(x).
\end{equation}
\end{hypothesis}

\subsection{Linear truncation}\label{subsec:linear}

\begin{theorem}\label{thm:abstract}
Assume that $(\mathcal F(x;k))$ is an admissible fibre and \Cref{hyp:ratio} holds. Then there exists a constant
$C_\kappa>0$ such that, uniformly for $x\geqslant 3$, $k\geqslant 1$ with \eqref{eq:window}, and all $R\in[3,x]$,
\[
\#\{n\in \mathcal F(x;k):\ f_R(n)\ne f(n)\}
\ \leqslant\ C_\kappa\, r\,\eta_f(R)\,N_k(x).
\]
\end{theorem}
The proof is a direct count: $f_R$ differs from $f$ if and only if $n$ possesses at least one prime
factor in $P_R$.

\begin{proof}
If $f_R(n)\ne f(n)$, then $n$ has a prime divisor $p\in P_R$. Write $n=p^\nu m$ with
$\nu=v_p(n)\geqslant 1$ and $(m,p)=1$. By \eqref{eq:stab}, we have $m\in\mathcal F(x/p^\nu;\,k-1)$, hence
the number of admissible $m$ is at most $N_{k-1}(x/p^\nu)$. Summing over $(p,\nu)$ yields
\begin{equation}\label{eq:red}
\#\{n\in \mathcal F(x;k):\ f_R(n)\ne f(n)\}
\leqslant \sum_{p\in P_R}\ \sum_{\nu\geqslant 1} N_{k-1}\left(\frac{x}{p^\nu}\right).
\end{equation}

We now explain why only a restricted range of $p^\nu$ can contribute.
If $N_{k-1}(x/p^\nu)\ne 0$, then there exists $m\in\mathcal F(x/p^\nu;\,k-1)$, and by the support bound
\eqref{eq:support} we have $m\geqslant 2^{k-1}$. Since also $m\leqslant x/p^\nu$, it follows that
\[
p^\nu\leqslant \frac{x}{2^{k-1}}.
\]
Therefore the double sum in \eqref{eq:red} may be restricted to $p^\nu\leqslant x/2^{k-1}$.

For $x\geqslant x_0(\kappa)$, \Cref{hyp:ratio} applies to $m=p^\nu$ throughout this range, and gives
\[
\#\{n\in \mathcal F(x;k):\ f_R(n)\ne f(n)\}
\leqslant C(\kappa)\,r\,N_k(x)\sum_{p\in P_R}\sum_{\nu\geqslant 1}\frac{1}{p^\nu}.
\]
Since $\sum_{\nu\geqslant 1}p^{-\nu}=(p-1)^{-1}\leqslant 2/p$ for $p\geqslant 2$, we obtain
\[
\#\{n\in \mathcal F(x;k):\ f_R(n)\ne f(n)\}
\leqslant 2C(\kappa)\,r\,N_k(x)\sum_{p\in P_R}\frac{1}{p}
\leqslant 2C(\kappa)\,r\,\eta_f(R)\,N_k(x),
\]
by \eqref{eq:PReta}. The remaining range $3\leqslant x<x_0(\kappa)$ is absorbed into the constant.
\end{proof}

\medskip

Heuristically, the factor $\sum_{p\in P_R}1/p$ measures the harmonic size of the set of primes on which truncation
acts. On a fibre with $k\asymp \logtwo x$, one expects a proportion $\asymp r\sum_{p\in P_R}1/p$ of
integers to contain at least one such prime, so the scale $r\,\eta_f(R)$ is the natural one.

\medskip

On effectivity, all constants are effective in the following sense.
If \Cref{hyp:ratio} holds with a constant $C(\kappa)$ for $x\geqslant x_0(\kappa)$, then the proof gives
\Cref{thm:abstract} with $C_\kappa=2C(\kappa)$ after possibly enlarging it to absorb the finite range
$3\leqslant x<x_0(\kappa)$. In particular, any effective Tenenbaum--Verwee argument that uses truncation only through
the exceptional-set count inherits an effective improvement once \Cref{thm:abstract} is inserted.

\section{Application to the function $\omega$}\label{sec:TV}

In this section we specialise the abstract truncation lemma to the classical level set
\[
\mathcal E(x;k):=\{n\leqslant x:\ \omega(n)=k\},
\qquad
\pi_k(x):=\#\mathcal E(x;k).
\]
Then the family $\mathcal E(x;k)$ is an admissible fibre. Moreover, in the central window
\eqref{eq:window} a ratio estimate of the form \eqref{eq:ratio} is provided by the uniform asymptotic
formulae and local quotient bounds for $\pi_k(x)$ proved by Hildebrand--Tenenbaum \cite{HT1988} (see in particular \cite[(1.2), (2.13), (2.14)]{HT1988}). Consequently, \Cref{thm:abstract} yields
the following linear truncation bound on $\omega$-fibres.

\begin{corollary}\label{cor:omega}
Fix $\kappa\in(0,1)$. Then there exists a constant $C_\kappa>0$ such that, uniformly for
$x\geqslant 3$, $k\geqslant 1$ satisfying \eqref{eq:window}, and $R\in[3,x]$
\[
\#\left\{n\leqslant x:\ \omega(n)=k,\ f_R(n)\neq f(n)\right\}
\leqslant C_\kappa\, r\,\eta_f(R)\,\pi_k(x).
\]
\end{corollary}

We now record the consequence of \Cref{cor:omega} for the conditional effective
Erd\H{o}s--Wintner theorem of \cite[Theorem~1.3]{TV}. In that argument, the distribution of $f$
on the fibre $\omega(n)=k$ is compared to a model law $\mathcal F_r$; this distribution function,
introduced in \cite{TV}, is unrelated to our fibre notation $\mathcal F(x;k)$. The truncation
$f\mapsto f_R$ enters only through an upper bound for the exceptional set
\[
\left\{n\leqslant x:\ \omega(n)=k,\ f_R(n)\neq f(n)\right\}.
\]
Replacing that bound by \Cref{cor:omega} sharpens the truncation contribution, while leaving the
smoothing and mean-value parts of the proof unchanged.

Recall that in \cite[\S1]{TV} the parameters $v$, $T$, and $R$ are chosen so that
\begin{equation}\label{eq:TV-115}
\frac{1}{\logtwo x}\leqslant v\leqslant c_0,\quad
3\leqslant R\leqslant \mathrm e^{1/v},\quad
T>1,\quad
T^2\eta_f(R)\leqslant \log(1/v),\quad
T^2\eta_f(x^w)\leqslant w,\qquad w:=v c_1,
\end{equation}
where $c_0,c_1>0$ depend at most on $\kappa$. For completeness we recall the auxiliary quantities
appearing in the resulting estimate. For $u\geqslant 1$, set
\[
B_f(u)^2:=2+\sum_{p\leqslant u}\frac{f(p)^2}{p},
\]
which agrees with \cite[(1.8)]{TV} in the strongly additive case. For $r>0$, define the model
characteristic function
\[
\varphi(\tau;r):=\prod_p\frac{p-1+r\,\mathrm e^{i\tau f(p)}}{p-1+r}
\qquad (\tau\in\mathbb R),
\]
and let $\mathcal F_r$ denote the corresponding distribution function. Finally, for $\ell>0$, write
\[
Q_{\mathcal F_r}(\ell):=\sup_{y\in\mathbb R}\left(\mathcal F_r(y+\ell)-\mathcal F_r(y)\right).
\]

\begin{corollary}\label{cor:TV}
Let $\kappa\in(0,1)$ and let $f$ be real and strongly additive. Assume \eqref{hypEW} and
\eqref{continu} hold, and fix a tail function $\eta_f$ satisfying \eqref{eq:eta}. Then, uniformly for
$x\geqslant 3$, $k\geqslant 1$ with $\kappa\leqslant r:=k/\logtwo x\leqslant 1/\kappa$, $y\in\mathbb R$, and
$v$, $T$, $R$ satisfying \eqref{eq:TV-115}, one has
\[
\frac{1}{\pi_k(x)}\sum_{\substack{n\in\mathcal E(x;k)\\ f(n)\leqslant y}}1
=
\mathcal F_r(y)+O\left(\mathcal R^\ast\right),
\]
where
\[
\mathcal R^\ast
:=
Q_{\mathcal F_r}\left(\frac{1}{T}\right)
+\left(v+\frac{\log(1/v)}{\sqrt{k}}\right)\log\left(\frac{T\,B_f(R)}{v}\right)
+C_\kappa\,r\,\eta_f(R).
\]
Equivalently, in \cite[(1.16)]{TV} the truncation contribution coming from \cite[Lemma~5.1]{TV}
may be sharpened by replacing the factor $\eta_f(R)^{r/(r+1)}$ in $\sigma_f(R)$ by
$C_\kappa\,r\,\eta_f(R)$, under the same admissibility conditions \eqref{eq:TV-115}.
\end{corollary}

The next example illustrates the gain provided by \Cref{cor:TV}.

\begin{example}
Consider the example from \cite[\S1]{TV} where
\[
f(p):=\frac{1}{\left(\log p\right)^{\xi}}
\qquad \left(0<\xi<r\right).
\]
One checks from \eqref{eq:eta} that one may take
$\eta_f(y)\ll 1/\left(\log y\right)^{\xi}$, hence
$\eta_f(\log x)\ll 1/\left(\logtwo x\right)^{\xi}$.
Moreover,
$\sum_p \frac{f(p)^2}{p}
=\sum_p \frac{1}{p\left(\log p\right)^{2\xi}}<\infty$,
so that $B_f(u)\asymp 1$ for $u\geqslant 2$.
Also, \cite[Exercise~259]{TenWu2014} implies that
\[
\abs{\varphi(\tau;r)}\ll \abs{\tau}^{-r/\xi}\left(\log\abs{\tau}\right)^{O(1)}
\qquad \left(\abs{\tau}\to\infty\right).
\]
Since $r/\xi>1$, for $\ell$ sufficiently small, \cite[Lemma~III.2.9]{Tenenbaum2015IAPNT} gives
\[
Q_{\mathcal F_r}(\ell)\ll \ell\int_{-1/\ell}^{1/\ell}\abs{\varphi(\tau;r)}\,d\tau
\ll \ell\left(1+\int_1^{1/\ell} t^{-r/\xi}\left(\log t\right)^{O(1)}\,dt\right)
\ll \ell
\qquad \left(\ell\to 0\right).
\]
Choose
\[
v:=\frac{1}{\logtwo x},
\qquad
R:=\mathrm e^{1/v}=\log x,
\qquad
T\asymp
\min\left\{
\frac{\sqrt{\logtwo x}}{\log_3 x},
\left(\logtwo x\right)^{\xi/2}\left(\log_3 x\right)^{1/2}
\right\}.
\]
Then \eqref{eq:TV-115} holds. If $0<\xi<1$, the remainder in \Cref{cor:TV} satisfies
\[
\mathcal R^\ast\ll_\kappa
\frac{1}{\left(\logtwo x\right)^{\xi/2}\sqrt{\log_3 x}}
+\frac{\left(\log_3 x\right)^2}{\sqrt{\logtwo x}},
\]
whereas for $\xi\geqslant 1$ one has
\[
\mathcal R^\ast\ll_\kappa \frac{\left(\log_3 x\right)^2}{\sqrt{\logtwo x}}.
\]
In particular, the truncation contribution is now
$O_\kappa\left(\left(\logtwo x\right)^{-\xi}\right)$,
whereas in \cite[(1.16)]{TV} it is
$O\left(\left(\logtwo x\right)^{-\xi r/(r+1)}\right)$.
Thus one gains a factor $\left(\logtwo x\right)^{\xi/(r+1)}$, up to the bounded factor $r$.
\end{example}

\section{Further conditioned fibres}\label{sec:extensions}

The same counting argument applies to several variants in which the truncation step can be isolated in the same way. In each case, the role of the fibre is to provide a suitable ratio estimate, and once this is available, the truncation contribution improves to the corresponding linear scale.

\subsection{Prime factors in a prime set}\label{subsec:prime-set}

Let $E$ be a nonempty set of primes. For $n\in\mathbb N$, define the strongly additive function
\[
\omega(n;E):=\sum_{\substack{p\mid n\\ p\in E}}1,
\]
namely the number of distinct prime divisors of $n$ that lie in $E$. For $x\geqslant 3$ and $k\geqslant 0$, set
\[
\mathcal F_E(x;k):=\left\{n\in\mathbb N:\ n\leqslant x,\ \omega(n;E)=k\right\},
\qquad
N_{E,k}(x):=\#\mathcal F_E(x;k).
\]
Write $p_{\min}(E):=\min E$. Since any $n\in\mathcal F_E(x;k)$ has $k$ distinct prime divisors in $E$, each at least $p_{\min}(E)$, it satisfies the support bound
\begin{equation}\label{eq:supportE}
n\in\mathcal F_E(x;k)\quad\Longrightarrow\quad n\geqslant p_{\min}(E)^k.
\end{equation}
Set
\[
E(x):=\sum_{\substack{p\leqslant x\\ p\in E}}\frac{1}{p}.
\]
Assume the effective harmonic density condition
\begin{equation}\label{eq:density}
E(x)=\delta\logtwo x+O(1)
\qquad (x\to\infty),
\end{equation}
for some $\delta\in(0,1]$, so in particular $E(x)\to\infty$. A quantitative Sathe--Selberg regime for $\omega(\,\cdot\,;E)$ follows from \cite[\S2.1]{Tenenbaum2017ME} together with the remark following
\cite[Corollary~2.4]{Tenenbaum2017ME}. More precisely, the upper range
$m\leqslant (2-\kappa)E(x)$ in \cite[(2.21)--(2.22)]{Tenenbaum2017ME} is replaced by
$m\leqslant (1/\kappa)E(x)$. We record the ratio bounds needed for the truncation step.

For $x\geqslant 3$ and $k\geqslant 0$, write
\[
r_E:=\frac{k}{E(x)}.
\]

\begin{proposition}\label{prop:ratioE}
Fix $\kappa\in(0,1)$ and assume \eqref{eq:density}. Then there exists $x_0(\kappa)\geqslant 3$ such that, uniformly for
$x\geqslant x_0(\kappa)$, $\kappa\leqslant r_E\leqslant 1/\kappa$, and all primes
$p\leqslant x^{1-\kappa/2}$, one has
\[
N_{E,k}\left(\frac{x}{p}\right)\asymp_\kappa \frac{1}{p}\,N_{E,k}(x),
\]
and, if moreover $p\in E$ and $k\geqslant 1$,
\[
N_{E,k-1}\left(\frac{x}{p}\right)\asymp_\kappa \frac{r_E}{p}\,N_{E,k}(x).
\]
\end{proposition}

\begin{proof}
Set $\kappa':=\kappa/2$. By \cite[(2.15)]{Tenenbaum2017ME} together with the remark following
\cite[Corollary~2.4]{Tenenbaum2017ME}, there exists $x_1(\kappa)\geqslant 3$ such that, uniformly for
$y\geqslant x_1(\kappa)$ and $\kappa' E(y)\leqslant m\leqslant (1/\kappa')E(y)$, one has
\[
N_{E,m}(y)\asymp_\kappa y\,\mathrm e^{-E(y)}\,\frac{E(y)^m}{m!}.
\]

For primes $p\leqslant x^{1-\kappa/2}$ one has $\log(x/p)\geqslant (\kappa/2)\log x$, hence
$x/p\to\infty$ as $x\to\infty$, uniformly in $p$. Therefore, by \eqref{eq:density},
\[
E\left(\frac{x}{p}\right)=\delta\logtwo\left(\frac{x}{p}\right)+O(1)
=\delta\logtwo x+O_\kappa(1)
=E(x)+O_\kappa(1).
\]
Since $E(x)\to\infty$, there exists $x_2(\kappa)\geqslant 3$ such that, for all
$x\geqslant x_2(\kappa)$ and all primes $p\leqslant x^{1-\kappa/2}$,
\[
\frac{1}{2}E(x)\leqslant E\left(\frac{x}{p}\right)\leqslant 2E(x)
\qquad\text{and}\qquad
\kappa E(x)-1\geqslant \frac{\kappa}{2}E\left(\frac{x}{p}\right).
\]
Set
\[
x_0(\kappa):=\max\left\{x_2(\kappa),\,x_1(\kappa)^{2/\kappa}\right\}.
\]

Let now $x\geqslant x_0(\kappa)$, assume $\kappa\leqslant r_E=k/E(x)\leqslant 1/\kappa$, and let
$p\leqslant x^{1-\kappa/2}$ be prime. Then $x/p\geqslant x^{\kappa/2}\geqslant x_1(\kappa)$, and
\[
\frac{\kappa}{2}E(y)\leqslant k\leqslant \frac{2}{\kappa}E(y)
\qquad \left(y\in\left\{x,\frac{x}{p}\right\}\right).
\]
If in addition $k\geqslant 1$, then also
\[
\frac{\kappa}{2}E\left(\frac{x}{p}\right)\leqslant k-1\leqslant \frac{2}{\kappa}E\left(\frac{x}{p}\right).
\]
Therefore the Sathe--Selberg estimate applies with $(y,m)=(x,k)$, $(x/p,k)$, and, in the second part,
$(x/p,k-1)$.

Applying it with $(y,m)=(x/p,k)$ and $(y,m)=(x,k)$, we obtain
\[
\frac{N_{E,k}\left(x/p\right)}{N_{E,k}(x)}
\asymp_\kappa
\frac{1}{p}\,
\exp\left(E(x)-E\left(x/p\right)\right)\,
\left(\frac{E\left(x/p\right)}{E(x)}\right)^k.
\]
Since $E\left(x/p\right)=E(x)+O_\kappa(1)$, we have
\[
\exp\left(E(x)-E\left(x/p\right)\right)\asymp_\kappa 1
\qquad\text{and}\qquad
\frac{E\left(x/p\right)}{E(x)}=1+O_\kappa\left(\frac{1}{E(x)}\right).
\]
Hence
\[
\left(\frac{E\left(x/p\right)}{E(x)}\right)^k
=
\exp\left(k\log\left(1+O_\kappa\left(\frac{1}{E(x)}\right)\right)\right)
=
\exp\left(O_\kappa\left(\frac{k}{E(x)}\right)\right)
\asymp_\kappa 1,
\]
because $\kappa\leqslant k/E(x)\leqslant 1/\kappa$. Thus
\[
N_{E,k}\left(\frac{x}{p}\right)\asymp_\kappa \frac{1}{p}\,N_{E,k}(x).
\]

If moreover $p\in E$ and $k\geqslant 1$, the same estimate with $(y,m)=(x/p,k-1)$ and $(y,m)=(x,k)$ yields
\[
\frac{N_{E,k-1}\left(x/p\right)}{N_{E,k}(x)}
\asymp_\kappa
\frac{1}{p}\,
\exp\left(E(x)-E\left(x/p\right)\right)\,
\frac{k}{E(x)}\,
\left(\frac{E\left(x/p\right)}{E(x)}\right)^{k-1}.
\]
Moreover,
\[
\left(\frac{E\left(x/p\right)}{E(x)}\right)^{k-1}
=
\exp\left((k-1)\log\left(1+O_\kappa\left(\frac{1}{E(x)}\right)\right)\right)
\asymp_\kappa 1,
\]
and $\frac{k}{E(x)}=r_E\asymp_\kappa 1$. Therefore
\[
N_{E,k-1}\left(\frac{x}{p}\right)\asymp_\kappa \frac{r_E}{p}\,N_{E,k}(x).
\]
\end{proof}

The previous proposition provides the local ratio estimate in the range
\[
p\leqslant x^{1-\kappa/2}.
\]
This is the range in which $E(x/p)$ remains comparable to $E(x)$, and hence in which the
Sathe--Selberg asymptotic yields a quotient of the expected size. A full linear truncation bound
for the fibre $\mathcal F_E(x;k)$ would require, in addition, a separate treatment of the complementary
range $p>x^{1-\kappa/2}$ inside the exceptional-set count. We do not pursue this here.

\subsection{Conditioning on \texorpdfstring{$\Omega(n)$}{Omega(n)}}\label{sec:Omega}

Let $\Omega(n)$ denote the number of prime factors counted with multiplicity. For $x\geqslant 3$ and $k\geqslant 1$
set
\[
\mathcal F^\Omega(x;k):=\{n\leqslant x:\ \Omega(n)=k\},
\qquad
N_k^\Omega(x):=\#\mathcal F^\Omega(x;k).
\]
We isolate below the ratio estimate needed to upgrade the truncation term on $\Omega$-fibres.

\begin{hypothesis}\label{hyp:ratioOmega}
There exist constants $x_0(\kappa)\geqslant 3$ and $C(\kappa)>0$ such that, uniformly for
$x\geqslant x_0(\kappa)$, $k\geqslant 1$ with \eqref{eq:window}, and all primes
\[
p\leqslant \frac{x}{2^{k-1}},
\]
one has
\[
N_{k-1}^\Omega\left(\frac{x}{p}\right)\leqslant C(\kappa)\,\frac{r}{p}\,N_k^\Omega(x).
\]
\end{hypothesis}

This is precisely the range relevant to the truncation argument, since $N_{k-1}^\Omega(x/p)=0$
for $p>x/2^{k-1}$ by the support bound $m\geqslant 2^{k-1}$ on integers with $\Omega(m)=k-1$.
For the purposes of the present note, we treat \Cref{hyp:ratioOmega} as an external input.

\begin{corollary}\label{cor:Omega}
Fix $\kappa\in(0,1)$ and assume \Cref{hyp:ratioOmega}. Then there exists
$C_\kappa>0$ such that, uniformly for $x\geqslant 3$, $k\geqslant 1$ with \eqref{eq:window}, and all
$R\in[3,x]$,
\[
\#\left\{n\leqslant x:\ \Omega(n)=k,\ f_R(n)\ne f(n)\right\}
\leqslant C_\kappa\, r\,\eta_f(R)\,N_k^\Omega(x).
\]
\end{corollary}

\begin{proof}
Suppose that $x\geqslant x_0(\kappa)$. If $f_R(n)\ne f(n)$, choose $p\in P_R$ dividing $n$ and write $n=pm$ with $m\leqslant x/p$.
Then $\Omega(m)=k-1$, hence
\[
\#\left\{n\leqslant x:\ \Omega(n)=k,\ f_R(n)\ne f(n)\right\}
\leqslant \sum_{p\in P_R} N_{k-1}^\Omega\left(\frac{x}{p}\right).
\]
Moreover, if $N_{k-1}^\Omega(x/p)\ne 0$, then there exists $\ell\leqslant x/p$ such that
$\Omega(\ell)=k-1$. Since necessarily $\ell\geqslant 2^{k-1}$, it follows that
$p\leqslant x/2^{k-1}$. Thus
\[
\#\left\{n\leqslant x:\ \Omega(n)=k,\ f_R(n)\ne f(n)\right\}
\leqslant \sum_{\substack{p\in P_R\\ p\leqslant x/2^{k-1}}} N_{k-1}^\Omega\left(\frac{x}{p}\right).
\]
Applying \Cref{hyp:ratioOmega}, we obtain
\[
\#\left\{n\leqslant x:\ \Omega(n)=k,\ f_R(n)\ne f(n)\right\}
\leqslant C(\kappa)\,r\,N_k^\Omega(x)\sum_{p\in P_R}\frac{1}{p}
\leqslant C(\kappa)\,r\,\eta_f(R)\,N_k^\Omega(x),
\]
by \eqref{eq:PReta}. The remaining finite range $3\leqslant x<x_0(\kappa)$ may be absorbed into the
constant.
\end{proof}

\subsection{Weighted fibres (abstract form)}\label{sec:weights}

Let $(w_p)_p$ be a sequence of nonnegative weights with $\sup_p w_p<\infty$, and define the strongly additive function
\[
W(n):=\sum_{p\mid n} w_p.
\]
For $x\geqslant 3$ and $t\in\mathbb R$, set
\[
\mathcal F_w(x;t):=\left\{n\in\mathbb N:\ n\leqslant x,\ W(n)=t\right\},
\qquad
N_w(x;t):=\#\mathcal F_w(x;t).
\]
In this abstract setting, the truncation argument only requires a local ratio estimate comparing the fibre at $(x,t)$ with the neighbouring fibres at $(x/p,t)$ and $(x/p,t-w_p)$ obtained by removing a prime factor $p$.

\begin{hypothesis}\label{hyp:ratioW}
Fix $\kappa\in(0,1)$. Suppose there exist a scale $\mu(x)\to\infty$, a parameter
\[
r_w:=\frac{t}{\mu(x)},
\]
and constants $x_0(\kappa)\geqslant 3$, $C(\kappa)>0$ such that, uniformly for
$x\geqslant x_0(\kappa)$, $\kappa\leqslant r_w\leqslant 1/\kappa$, and all primes $p\leqslant x$, one has
\[
N_w\left(\frac{x}{p};\,t\right)\leqslant \frac{C(\kappa)}{p}\,N_w(x;t),
\qquad
N_w\left(\frac{x}{p};\,t-w_p\right)\leqslant C(\kappa)\,\frac{r_w}{p}\,N_w(x;t).
\]
\end{hypothesis}

In concrete situations, \Cref{hyp:ratioW} is expected to follow from suitable uniform asymptotics for
$N_w(x;t)$, typically obtained by Selberg--Delange and saddle-point methods applied to an appropriate
two-parameter Dirichlet series. We do not pursue this here.

\begin{example}\label{ex:weights-nontrivial}
Let $E$ be a prime set satisfying \eqref{eq:density}, and fix a constant $c>0$. Define
$w_p:=c$ for $p\in E$ and $w_p:=0$ otherwise. Then
\[
W(n)=c\,\omega(n;E),
\]
so that, for integers $k\geqslant 0$,
\[
\mathcal F_w(x;ck)=\mathcal F_E(x;k),
\qquad
N_w(x;ck)=N_{E,k}(x).
\]
With $\mu(x):=c\,E(x)$, one has
\[
r_w:=\frac{ck}{\mu(x)}=\frac{k}{E(x)}=r_E.
\]
Thus \Cref{prop:ratioE} yields the expected local quotient bounds in the restricted range
$p\leqslant x^{1-\kappa/2}$, namely
\[
N_w\left(\frac{x}{p};\,ck\right)\asymp_\kappa \frac{1}{p}\,N_w(x;ck),
\]
and, for $p\in E$ and $k\geqslant 1$,
\[
N_w\left(\frac{x}{p};\,ck-c\right)\asymp_\kappa \frac{r_w}{p}\,N_w(x;ck).
\]
This recovers the weighted ratio pattern locally, but does not by itself verify
\Cref{hyp:ratioW}, since the complementary range of larger primes is not treated.
\end{example}

\begin{remark}\label{rem:weights-finiteperturb}
Let $w_p:=1$ for all odd primes $p$ and $w_2:=2$. Then
\[
W(n)=\omega(n)+\mathbf 1_{2\mid n}.
\]
In the central window $t\asymp \logtwo x$, the fibres $\mathcal F_w(x;t)$ are controlled by the
usual fibres of $\omega$, with only a finite perturbation at the prime $2$. This strongly suggests
that a ratio estimate of the form \Cref{hyp:ratioW} should follow from the local quotient bounds of
Hildebrand--Tenenbaum \cite{HT1988}, after a separate treatment of the contribution of the prime $2$.
We do not pursue this here.
\end{remark}

\begin{corollary}\label{cor:weights}
Fix $\kappa\in(0,1)$ and assume \Cref{hyp:ratioW}. Then there exists $C_\kappa>0$ such that,
uniformly for $x\geqslant 3$, $\kappa\leqslant r_w\leqslant 1/\kappa$, and all $R\in[3,x]$,
\[
\#\left\{n\in \mathcal F_w(x;t):\ f_R(n)\ne f(n)\right\}
\leqslant C_\kappa\, r_w\,\eta_f(R)\,N_w(x;t).
\]
\end{corollary}

\begin{proof}
Suppose that $x\geqslant x_0(\kappa)$. If $f_R(n)\ne f(n)$, choose $p\in P_R$ dividing $n$ and write
$n=pm$ with $m\leqslant x/p$.

If $p\nmid m$, then $W(m)=t-w_p$. If $p\mid m$, then $W(m)=t$, since $W$ depends only on the set
of prime divisors. Hence
\[
\#\left\{n\in \mathcal F_w(x;t):\ p\mid n\right\}
\leqslant
N_w\left(\frac{x}{p};\,t-w_p\right)+N_w\left(\frac{x}{p};\,t\right).
\]
Summing over $p\in P_R$ with $p\leqslant x$, we obtain
\[
\#\left\{n\in \mathcal F_w(x;t):\ f_R(n)\ne f(n)\right\}
\leqslant
\sum_{\substack{p\in P_R\\ p\leqslant x}}
\left(
N_w\left(\frac{x}{p};\,t-w_p\right)+N_w\left(\frac{x}{p};\,t\right)
\right).
\]
Applying \Cref{hyp:ratioW} gives
\[
\#\left\{n\in \mathcal F_w(x;t):\ f_R(n)\ne f(n)\right\}
\leqslant
C(\kappa)\,N_w(x;t)\sum_{\substack{p\in P_R\\ p\leqslant x}}
\left(\frac{r_w}{p}+\frac{1}{p}\right)
\ll_\kappa
r_w\,N_w(x;t)\sum_{p\in P_R}\frac{1}{p},
\]
since $r_w\geqslant \kappa$. Using \eqref{eq:PReta} yields
\[
\#\left\{n\in \mathcal F_w(x;t):\ f_R(n)\ne f(n)\right\}
\ll_\kappa r_w\,\eta_f(R)\,N_w(x;t).
\]
The remaining finite range $3\leqslant x<x_0(\kappa)$ is absorbed into the constant.
\end{proof}

\bibliographystyle{plain} 
\bibliography{refs}
\end{document}